\documentclass[a4paper,10pt]{amsart}
\usepackage{amsmath,amssymb,amsthm,amscd}
\usepackage[all]{xy}
\renewcommand{\iff}{if and only if }
\newcommand{\st}{such that }

\newcommand{\ModR}{\hbox{{\rm Mod-}}R}

\newcommand{\Z}{\mathbb{Z}}

\DeclareMathOperator{\Hom}{Hom}

\DeclareMathOperator{\Ker}{Ker}
\DeclareMathOperator{\Img}{Im}

\usepackage[usenames]{color}

\theoremstyle{plain}
\newtheorem{thm}{Theorem}[section]
\newtheorem{prop}[thm]{Proposition}
\newtheorem{lem}[thm]{Lemma}

\newtheorem{cor}[thm]{Corollary}
\theoremstyle{definition}

\theoremstyle{remark}

\begin{document}
\title[On the solvability of systems of linear equations over $\Z$]{On the nontrivial solvability of systems of homogeneous linear equations over $\mathbb Z$ in ZFC}

\author{\textsc{Jan \v Saroch}}
\address{Charles University, Faculty of Mathematics and Physics, Department of Algebra \\ 
Sokolovsk\'{a} 83, 186 75 Praha~8, Czech Republic}
\email{saroch@karlin.mff.cuni.cz}


\thanks{This work has been supported by Charles Univ.\ Research Centre program No.\ UNCE/SCI/022 and by grant GA\v CR 17-23112S}

\date{\today}

\begin{abstract} Following the paper \cite{HeTa}, we investigate in ZFC the following compactness question: for which unountable cardinals $\kappa$, an arbitrary nonempty system $S$ of homogeneous $\mathbb Z$-linear equations is nontrivially solvable in $\mathbb Z$ provided that each its nonempty subsystem of cardinality $<\kappa$ is nontrivially solvable in $\mathbb Z$? \end{abstract}

\maketitle
\vspace{4ex}

\section{Introduction and preliminaries}

In what follows, group means always an abelian group, i.e. a $\mathbb Z$-module. We say that a system $S$ of homogeneous $\mathbb Z$-linear equations with a set $X=\{x_i\mid i\in I\}$ of variables is \emph{nontrivially solvable} in a group $H$ if there exists a mapping $f:X\to H\setminus\{0\}$ such that, whenever $\sum_{j\in J}a_jx_j = 0$ is an equation from $S$ (where $J$ is a~finite subset of $I$ and $a_j\in \mathbb Z$ for each $j\in J$), then $\sum_{j\in J}a_jf(x_j) = 0$ holds in $H$.

This notion of nontriviality is a little bit unusual. If we assume instead that the mapping $f$ goes to $H$ and it is not constantly zero on all $x\in X$ that appear in the system $S$, we say that the system $S$ is \emph{weakly nontrivially solvable} in $H$. More natural as it might me, this weaker notion has got one significant disadvantage: unlike with nontrivial solvability, if a system $S$ is weakly nontrivially solvable and $T$ is a nonempty subsystem of $S$, then $T$ need not be weakly nontrivially solvable. Notice also that an empty system $S$ is (weakly) nontrivially solvable by definition.
\smallskip

Motivated by the work in \cite{HeTa}, our aim is to characterize the class $\mathcal S$ ($\mathcal {WS}$, resp.) of all infinite cardinals $\kappa$ such that any system $S$ of homogeneous $\mathbb Z$-linear equations is nontrivially (weakly nontrivially, resp.) solvable in $\mathbb Z$ provided that each subsystem $T\subseteq S$ of cardinality $<\kappa$ is nontrivially (weakly nontrivially, resp.) solvable in $\mathbb Z$. In \cite[Section 2.2]{HeTa}, the authors present several well-known examples of countable $S$ which show in ZF that $\aleph_0\not\in\mathcal S\cup\mathcal{WS}$. They also discuss various interesting related questions in ZF: among other things, they provide a model of ZF without choice where $\aleph_1\not\in\mathcal{S}$ while they note that the result is not known in ZFC.

In this short note, we use $\kappa$-free groups with trivial dual to show that ZFC actually proves $\aleph_n\not\in\mathcal{S}$ for each $n\in\omega$. Moreover, it is consistent with ZFC that $\mathcal S = \mathcal{WS} = \varnothing$. On the other hand, we are able to prove that $\kappa\in\mathcal{WS}\cap\mathcal S$ whenever there exits a regular $\mathcal L_{\omega_1\omega}$-compact cardinal below $\kappa$ (see Corollary~\ref{c:S_pos} and Theorem~\ref{t:WS_main}). A detailed discussion is contained in Section~\ref{s:conrem}.

For an unexplained terminology, we recommend, for instance, the very well-written extensive book \cite{EM}.

\smallskip

\section{The case of $\mathcal S$}

Recall that, given an uncountable cardinal $\nu$, we say that a cardinal $\kappa$ is \emph{$\mathcal L_{\nu\omega}$-compact} if every $\kappa$-complete filter on any set $I$ can be extended to a $\nu$-complete ultrafilter. Observe that a cardinal $\mu$ is $\mathcal L_{\nu\omega}$-compact whenever there exists an $\mathcal L_{\nu\omega}$-compact cardinal $\lambda$ such that $\lambda\leq\mu$. This is obviously a large cardinal notion for the existence of an $\mathcal L_{\nu\omega}$-compact cardinal implies the existence of a measurable cardinal.

The following result is a classic one. We give a proof for the reader's convenience.

\begin{prop} \label{p:S_pos} Let $\lambda$ be a regular $\mathcal L_{\nu\omega}$-compact cardinal and $\mathcal Z$ a structure in a~first order language $L$ with $|Z|<\nu$. Then a system $S$ consisting of first order formulas in variables from a set $X$ is realized in $\mathcal Z$ provided that each its subsystem $T$ of cardinality $<\lambda$ is realized in $\mathcal Z$.
\end{prop}

\begin{proof} First, let $E$ denote the set $Z^X$ of all mappings from $X$ to $Z$. By the assumption, for each $T\in [S]^{<\lambda}$, there exists $e\in E$ such that $\mathcal Z\models \varphi[e]$ for each $\varphi\in T$. Let $\mathcal F$ be the filter on $E$ generated by the sets $E_T = \{e\in E \mid \mathcal Z\models \varphi[e]\hbox{ for all }\varphi\in T\}$. Since $\lambda$ is regular, we see that $\mathcal F$ is a $\lambda$-complete filter. Let $\mathcal G$ denote an extension of $\mathcal F$ to a $\nu$-complete ultrafilter.

For each $(x,z)\in X\times Z$, put $E_{x,z} = \{e\in E \mid e(x) = z\}$ and define $f\in Z^X$ by the assignment $f(x) = z \Leftrightarrow E_{x,z}\in\mathcal G$. We can do it like that since the ultrafilter $\mathcal G$ picks, for each fixed $x\in X$, exactly one element from the disjoint partition $E = \bigcup_{z\in Z} E_{x,z}$; recall that $|Z|<\nu$.

Now let $\varphi\in S$ be arbitrary and $x_1,\dots, x_n$ be variables freely occuring in $\varphi$. Then $\varnothing\neq E_{\{\varphi\}}\cap\bigcap_{i=1}^n E_{x_i,f(x_i)}\in\mathcal G$, and so $f\in E_{\{\varphi\}}$. We conclude that $S$ is realized in $\mathcal Z$ using the evaluation $f$.
\end{proof}

\begin{cor} \label{c:S_pos} Let $\kappa$ be a cardinal and $\lambda\leq \kappa$ a regular $\mathcal L_{\omega_1\omega}$-cardinal. Then every nonempty system $S$ of homogeneous $\mathbb Z$-linear equations in variables from a set $X$ is nontrivially solvable in $\mathbb Z$ whenever each its nonempty subsystem of cardinality $<\kappa$ is nontrivially solvable in $\mathbb Z$. In other words $\kappa\in\mathcal S$.
\end{cor}

\begin{proof} In the system $S$ replace each equation $\psi$ in variables $x_1,\dots,x_n\in X$ by the formula $\psi\,\&\bigwedge_{i=1}^n x_i \neq 0$ and use Proposition~\ref{p:S_pos}.
\end{proof}

\smallskip

Before we turn our attention to the negative part, we need one preparatory lemma which holds in the general context of $R$-modules over a (commutative) noetherian domain. Recall, that, for a module $M\in\ModR$ and an ordinal number $\sigma$, an increasing chain $\mathcal M = (M_\alpha\mid \alpha\leq\sigma)$ of submodules of $M$ is called a \emph{filtration of}~$M$ if $M_0 = 0, M_\beta = \bigcup_{\alpha<\beta} M_\alpha$ whenever $\beta\leq\sigma$ is a limit ordinal, and $M_\sigma = M$.

\begin{lem}\label{l:prep} Let $R$ be an infinite noetherian domain, $M$ a free $R$-module of rank $\mu\geq\aleph_0$, and $\mathcal M = (M_\alpha\mid \alpha\leq\sigma)$ be a filtration of $M$ where, for all $\alpha<\sigma$, $M_{\alpha+1} = M_\alpha + \langle a_\alpha \rangle$ with $a_\alpha\in M\setminus M_\alpha$. For each $\alpha<\sigma$, let $z_\alpha\in R$ be arbitrary.

Then there is a homomorphism $\psi:M \to R$ such that $\psi (a_\alpha)\neq z_\alpha$ for all $\alpha<\sigma$.
\end{lem}

\begin{proof} First, assume that $\mu = \aleph_0$. Let $\{g_n\mid n<\omega\}$ be a set of free generators of~$M$. For each $\alpha<\sigma$, we express $a_\alpha$ as $\sum_{n\in I_\alpha} b_{n\alpha}g_n$, where $I_\alpha$ is a finite subset of~$\omega$ and $b_{n\alpha}\in R\setminus\{0\}$ for every $n\in I_\alpha$.

Using the fact that a free $R$-module of finite rank is noetherian, we infer that, for each $n<\omega$, the set $A_n = \{\alpha<\sigma\mid I_\alpha\subseteq \{0,1,\dots, n\}\}$ is finite. Note that $\sigma = \bigcup _{n<\omega}A_n$.
On the free generators of $M$, we recursively construct a~homomorphism $\psi:M \to R$ as follows:
\smallskip

Let $\psi (g_0)$ be arbitrary such that, for each $\alpha\in A_0$, $b_{0\alpha}\psi(g_0)\neq z_\alpha$. There is always an applicable choice by the hypothesis on $R$. Assume that $n>0$, $\psi(g_{n-1})$ is defined, and $\psi(a_\alpha)\neq z_\alpha$ for each $\alpha\in A_{n-1}$. 

We define $\psi(g_n)$ arbitrarily in such a way that, for each $\alpha\in A_n\setminus A_{n-1}$, we have $$b_{n\alpha}\psi(g_n) \neq z_\alpha - \sum_{k\in I_\alpha\setminus\{n\}} b_{k\alpha}\psi(g_k).$$
This is possible, since $A_n\setminus A_{n-1}$ is finite, $b_{n\alpha}\neq 0$ for each $\alpha$ from this set, and $R$ is an infinite domain. It immediately follows that $\psi(a_\alpha)\neq z_\alpha$ for each $\alpha\in A_n$.

\smallskip

Now, let $\mu$ be an uncountable cardinal. Again, let $\{g_\beta \mid \beta<\mu\}$ be a set of free generators of $M$, and put $G_B = \langle g_\beta \mid \beta \in B\rangle$ for all $B\subseteq\mu$.

We use ideas from \cite[Section 7.1]{GT}. First, we set $A_\alpha = \langle a_\alpha\rangle\leq M$. We say that a~subset $S$ of the ordinal $\sigma$ is \emph{`closed'} if every $\alpha\in S$ satisfies $$M_\alpha\cap A_\alpha \subseteq \sum _{\beta\in S,\beta<\alpha}A_\beta.$$

Notice that any ordinal $\alpha\leq\sigma$ is a `closed' subset of $\sigma$. For a `closed' subset $S$, we define $M(S) = \sum_{\alpha\in S}A_\alpha$. 
The results from \cite[Section 7.1]{GT} give us the following:

\begin{enumerate}
	\item For a system $(S_i\mid i\in I)$ of `closed' subsets, $\bigcap_{i\in I} S_i$ and $\bigcup_{i\in I} S_i$ is `closed' as well.
	\item For $S, S^\prime$ `closed' subsets of $\sigma$, we have $S\subseteq S^\prime \Longleftrightarrow M(S)\subseteq M(S^\prime)$.
	\item Let $S$ be a `closed' subset of $\sigma$ and $X$ be a countable subset of $M$. Then there is a `closed' subset $S^\prime$ such that $M(S)\cup X\subseteq M(S^\prime)$ and $|S^\prime\setminus S|<\aleph_1$.
\end{enumerate}

Using the properties listed above, we are going to construct a filtration $\mathcal N = (M(S_\alpha)\mid \alpha\leq\mu)$ of $M$ such that, for each $\alpha<\mu$, a) $S_\alpha$ is `closed', b) $S_{\alpha+1}\setminus S_\alpha$ is countable, and c) there exists $B_\alpha\subseteq\mu$ such that $G_{B_\alpha} = M(S_\alpha)$ and $\alpha\subseteq B_\alpha$.

We proceed by the transfinite recursion, starting with $S_0 = B_0 = \varnothing$. Let $S_\alpha$ and $B_\alpha$ be defined and $\alpha<\mu$. Then $|S_\alpha|+|B_\alpha|<\mu$ (using b) and c)). Let $B^0\supseteq B_\alpha\cup \{\alpha\}$ be any subset of $\mu$ with $|B^0\setminus B_\alpha| = \aleph_0$. By $(3)$, we find $S^0\supseteq S_\alpha$ such that $M(S^0)\supseteq G_{B^0}$ and $|S^0\setminus S_\alpha|<\aleph_1$. Assuming $B^n, S^n$ are defined for $n<\omega$, we can find $B^{n+1}\supseteq B^n$ with $|B^{n+1}\setminus B^n|<\aleph_1$ such that $G_{B^{n+1}}\supseteq M(S^n)$, and $S^{n+1}\supseteq S^n$ with $|S^{n+1}\setminus S^n|<\aleph_1$ such that $M(S^{n+1})\supseteq G_{B^{n+1}}$. Put $S_{\alpha+1} = \bigcup_{n<\omega} S^n$ and $B_{\alpha +1} = \bigcup_{n<\omega} B^n$. This completes the isolated step. In limit steps, we simply take unions. Since $M(S_\mu) = M$, we have $S_\mu = \sigma$ by $(2)$.

\smallskip

Now, for each $\alpha<\mu$, we have the countable sets $C_\alpha = B_{\alpha+1}\setminus B_\alpha$ and $T_\alpha = S_{\alpha+1}\setminus S_\alpha$, and the canonical projection $\pi_\alpha:M(S_{\alpha+1})\to G_{C_\alpha}$. Let $\tau$ be the ordinal type of $(T_\alpha,<)$, and fix an order-preserving bijection $i:\tau\to T_\alpha$.

Since $S_\alpha\cup(S_{\alpha+1}\cap \beta)$ is `closed' for any $\beta\leq\sigma$ by $(1)$, the part $(2)$ yields that the chain $(N_\beta \mid \beta\leq\tau)$ of modules defined as $N_\beta = M(S_\alpha\cup(S_{\alpha+1}\cap i(\beta)))$, for $\beta<\tau$, and $N_\tau = M(S_{\alpha+1})$ is strictly increasing. Notice that $N_0 = M(S_\alpha)$.

If we put $\bar{N}_\beta = \pi_\alpha[N_\beta]$ for all $\beta\leq\tau$, it follows that the strictly increasing chain $(\bar{N}_\beta \mid \beta\leq\tau)$ is a filtration of the free module $G_{C_\alpha}$ of countable rank. Moreover, for each $\beta<\tau$, we have $\bar{N}_{\beta+1} = \bar{N}_\beta + \langle\pi_\alpha(a_{i(\beta)})\rangle$.

Finally, we recursively define the homomorphism $\psi:M\to R$. Let $\alpha<\mu$ and assume that $\psi\restriction G_{B_\alpha}$ is constructed with the property $\psi(a_\gamma)\neq z_\gamma$ for all $\gamma\in S_\alpha$. By the already proven part for $\mu = \aleph_0$, we can define $\psi\restriction G_{C_\alpha}$ in such a way that $\psi(\pi_\alpha(a_\gamma))\neq z_\gamma-\psi(a_\gamma - \pi_\alpha(a_\gamma))$ for all $\gamma\in T_\alpha$; observe that the right-hand side of the inequality is already defined since $a_\gamma - \pi_\alpha(a_\gamma)\in G_{B_\alpha}$. We immediately get $\psi(a_\gamma)\neq z_\gamma$ for all $\gamma\in S_{\alpha+1}$.

\end{proof}

\medskip

For the negative part, we start with an uncountable cardinal $\kappa$ and a~$\kappa$-free group $G$ with the trivial dual, i.e. with the property $G^*:= \Hom(G,\mathbb Z) = 0$; here, $\kappa$-free means that any $<\!\kappa$-generated subgroup of $G$ is free. We will discuss the existence of such groups, as well as the question whether $G$ can be taken with $|G| = \kappa$, later on. Firstly, we show how the existence of such $G$ implies that $\kappa\not\in\mathcal S$.

\medskip

Let us denote by $\lambda$ the cardinality of $G$ and express $G$ as a quotient $F/K$ where $F$ is a free group of rank $\lambda$. Notice that $\lambda\geq \kappa$. Let $\pi:F\to F/K$ denote the canonical projection and let $\{e_\alpha\mid \alpha<\lambda\}$ be a set of free generators of the group~$F$. For each $A\subseteq\lambda$, let $F_A$ denote the subgroup of $F$ generated by $\{e_\alpha\mid \alpha\in A\}$. We can w.l.o.g. assume that $$\Img(\pi\restriction F_\beta) \subsetneq \Img(\pi\restriction F_{\beta +1})\hbox{ for each ordinal }\beta < \lambda.\eqno{(*)}$$

The group $K$ is also free of rank $\lambda$. If it had a smaller rank, $G$ would have possessed a free direct summand---a contradiction with $G^* = 0$. Let $\{k_\beta\mid \beta<\lambda\}$ denote a set of (free) generators of the group $K$. Consider the uncountable set $$S = \Biggl\{\sum_{\alpha\in J_\beta} a_{\alpha\beta}x_\alpha =~0 \mid \beta<\lambda, J_\beta\in [\lambda]^{<\omega}, (\forall \alpha\in J_\beta)(a_{\alpha\beta}\in\mathbb Z), \sum_{\alpha\in J_\beta} a_{\alpha\beta}e_\alpha = k_\beta\Biggr\}$$ of homogeneous $\mathbb Z$-linear equations with the set $\{x_\alpha\mid\alpha<\lambda\}$ of variables. We will show that this is the desired counterexample.

\medskip

First of all, $S$ does not have even a weakly nontrivial solution in $\mathbb Z$. Indeed, any such solution would define a nonzero homomorphism $\psi$ from $F$ to $\mathbb Z$ which is zero on $K$. Hence $\psi$ would provide for a nonzero homomorphism from $G$ to $\mathbb Z$, a contradiction.

On the other hand, we can show

\begin{prop} \label{p:S_neg} Any system $T\subseteq S$ with cardinality $<\kappa$ has a nontrivial solution in $\mathbb Z$.
\end{prop}

\begin{proof} Let $A\in [\lambda]^{<\kappa}$ be an infinite set such that whenever $x_\alpha$ appears in an equation from $T$ then $\alpha\in A$. Put $M = \Img(\pi\restriction F_A)$.

Since $G$ is $\kappa$-free, $M$ is a free group (of infinite rank). Let $\sigma$ denote the ordinal type of $A$ and fix an order-preserving bijection $i:\sigma\to A$. For each $\alpha\leq\sigma$, set $M_\alpha = \langle \pi(e_{i(\beta)}) \mid \beta<\alpha\rangle$. Then $(M_\alpha \mid \alpha\leq\sigma)$ is a filtration of $M$ such that $M_{\alpha+1} = M_\alpha + \langle\pi(e_{i(\alpha)})\rangle$ where $\pi(e_{i(\alpha)})\not\in M_\alpha$ for all $\alpha<\sigma$ (using $(*)$).

Applying Lemma~\ref{l:prep} with $R = \mathbb Z$ and $z_\gamma = 0$ for all $\gamma<\sigma$, we obtain a~homomorphism $\psi:M \to \mathbb Z$ such that $\psi(\pi(e_\alpha))\neq 0$ for all $\alpha\in A$. The assignment $x_\alpha\mapsto \psi(\pi(e_\alpha))$, $\alpha\in A$, is the desired nontrivial solution of the system $T$ in $\mathbb Z$.
\end{proof}

\begin{cor} \label{c:S_neg} Let $\kappa$ be an uncountable cardinal. If there exists a $\kappa$-free group $G$ with $G^* = 0$, then $\kappa\not\in\mathcal S\cup\mathcal {WS}$.
\end{cor}

\section{The case of $\mathcal{WS}$}

For the weaker notion of nontrivial solvability, we have the following general result.

\begin{prop} \label{p:WS_main} Let $\kappa$ be an uncountable cardinal. The following conditions are equivalent:
\begin{enumerate}
	\item There exists a regular cardinal $\lambda\leq\kappa$ which is $\mathcal L_{\omega_1\omega}$-compact.
  \item There is a regular cardinal $\lambda\leq\kappa$ such that each group $A\in\Ker\Hom(-,\Z)$ is the sum of its subgroups of cardinality $<\lambda$ which are contained in $\Ker\Hom(-,\Z)$.
	\item For any nonempty system $S$ of homogeneous $\Z$-linear equations \st $S$ has no weakly nontrivial solution in $\Z$, and any $C\in [S]^{<\kappa}$, there exists $T\in[S]^{<\kappa}$ such that $C\subseteq T$ and $T$ has no weakly nontrivial solution in $\Z$.
\end{enumerate}
\end{prop}

\begin{proof} The equivalence of $(1)$ and $(2)$ follows directly from \cite[Corollary 5.4]{BM}. Let us show that $(2)$ is equivalent to $(3)$. To this end, we are going to use the following two-way translation.

Given any system $S=\{k_j = 0 \mid j\in J\}$ of homogeneous $\Z$-linear equations with the set $X$ of variables, we can build a group $A = F/K$ where $F$ is freely generated by the elements of the set $X$ and $K$ is generated by the set $\{k_j\mid j\in J\}$. Then $\Hom(A,\Z) = 0$ \iff $S$ has no weakly nontrivial solution in $\Z$. On the other hand, for a given group $A$ and its presentation $F/K$ where $F$ is freely generated by a set $X$, the same equivalence holds for the system $S = \{k_j = 0 \mid j\in J\}$ of homogeneous $\Z$-linear equations where $\{k_j \mid j\in J\}$ is a fixed set of generators of $K$ expressed as $\Z$-linear combinations of elements from the set $X$.

Proving $(2) \Longrightarrow (3)$, we start with a system $S$ and a set $C\in [J]^{<\kappa}$.
Consider the group $A$ constructed for $S$ as in the previous paragraph, and let $Y_0$ denote the set of all the elements from $X$ appearing in equations $k_j = 0, j\in C$.

Let $\mu \geq \lambda$ be a regular uncountable cardinal such that $|C| < \mu \leq \kappa$. Since $\Ker\Hom(-,\Z)$ is closed under direct sums and quotients, and $\mu$ is regular, there exists, by $(2)$, $G_0\in\Ker\Hom(-,\Z)$ such that $G_0$ is a subgroup of $A$, $|G_0|<\mu$ and $Y_0+K := \{y+K\mid y\in Y_0\}\subseteq G_0$. Now, take any $Y_1\in [X]^{<\mu}, Y_0\subseteq Y_1$ such that:
\begin{enumerate}
	\item[$(a)$] $G_0$ is contained in the subgroup of $A$ generated by $Y_1 + K$.

	\item[$(b)$] There exists $C_0\in [J]^{<\mu}$ such that $\langle Y_0 \rangle \cap K$ is contained in the subgroup of $K$ generated by $\{k_j\mid j\in C_0\}$, and $Y_1$ contains all the elements from $X$ appearing in equations $k_j = 0, j\in C_0$.
\end{enumerate}
For this $Y_1$, we obtain, using $(2)$, a subgroup $G_1$ of $A$ with $|G_1|<\mu$, and so on.

After $\omega$ steps, we have the group $G = \sum_{n<\omega} G_n \in\Ker\Hom(-,\Z)$ generated by $Y+ K$ where $Y = \bigcup_{n<\omega} Y_n\in [X]^{<\mu}$. By the construction, we have also $G = \langle y+K\mid y\in Y\rangle \cong \langle Y\rangle/\langle k_j \mid j\in \bigcup_{n< \omega} C_n\rangle$. Finally, we put $T = \{k_j = 0 \mid j\in \bigcup_{n< \omega} C_n\}$.

\smallskip

Now, let us prove the implication $\neg(1) \Longrightarrow \neg(3)$. First, assume that $\kappa$ is not $\mathcal L_{\omega_1\omega}$-compact. Following \cite[Theorem 5.3]{BM} and its proof, we start with $A = \Z^I/\mathcal F$ where $\mathcal F$ is a $\kappa$-complete filter on $I$ which cannot be extended to an $\omega_1$-complete ultrafilter. From the latter part, it follows that $\Hom(A,\Z) = 0$. The $\kappa$-completeness of $\mathcal F$, on the other hand, assures that any subgroup of $A$ of cardinality $<\kappa$ can be embedded into $\Z^I$.

Consider a system $S$ of homogeneous $\Z$-linear equations associated to the group $A$ presented as $F/K$ where $F$ is freely generated by a set $X$. We can w.\,l.\,o.\,g. assume that no $x\in X$ is contained in $K$. Let $C\in [J]^{<\kappa}$ be non-empty. We shall show that the system $\{k_j = 0\mid j\in C\}$ has weakly nontrivial solution in $\Z$.

As in the proof of the other implication, we can possibly enlarge $C$ to some $D\subseteq J$ such that $|D|\leq |C| + \aleph_0$ and---denoting by $Y$ the set of all the elements from $X$ appearing in equations $k_j = 0, j\in D$---$\,\langle y+K\mid y\in Y\rangle \cong \langle Y\rangle/\langle k_j\mid j\in D\rangle$. Let us denote the latter group by $H$ and fix an embedding $i:H\to \Z^I$ (which exists since $|H|<\kappa$).

Let $y\in Y$ be any element appearing in (one of the) equations $k_j = 0, j\in C$. Since $i(y+K)\neq 0$ there is a projection $\pi:\Z^I \to \Z$ such that $\pi i(y+K)\neq 0$. The assignment $x\mapsto \pi i(x+K)$ defines the desired weakly nontrivial solution of the system $\{k_j = 0\mid j\in C\}$ in $\Z$.

\smallskip

It remains to tackle the possibility that $\kappa$ is the least $\mathcal L_{\omega_1\omega}$-compact cardinal and $\kappa$ is singular. We know by \cite{BM2} that $\gamma = cf(\kappa)$ is greater than or equal to the first measurable cardinal in this case. Let $(\kappa_\alpha\mid \alpha<\gamma)$ be an increasing sequence of cardinals $<\kappa$ converging to $\kappa$.

Consider the group $A = \bigoplus _{\alpha<\gamma} A_\alpha$ where, for each $\alpha<\gamma$, $A_\alpha\in\Ker\Hom(-,\Z)$ is not a sum of its subgroups of cardinality $<\kappa_\alpha$ which belong to $\Ker\Hom(-,\Z)$. Assume, for the sake of contradiction, that $(3)$ holds for the system $S$ of homogeneous $\Z$-linear equations associated to the group $A$ (more precisely, to its presentation $F/K$).

By the definition of $A$, there exists, for each $\alpha < \gamma$, an element $a_\alpha\in A$ such that $a_\alpha$ is not contained in any subgroup $H$ of $A$ of cardinality $<\kappa _\alpha$ with the property $\Hom (H,\Z) = 0$.

We know that there is $C_0\in [J]^{<\kappa}$ and $Y_0\subseteq X$ consisting of the elements from $X$ appearing in the equations $k_j = 0, j\in C_0$ such that $\{a_\alpha \mid \alpha<\gamma\}\subseteq \langle y+K \mid y\in Y_0\rangle \cong \langle Y_0 \rangle /\langle k_j \mid j\in C_0\rangle$.

For this $C_0$, we obtain a corresponding $T_0\in [J]^{<\kappa}$ using $(3)$. We continue by finding $C_1\in [J]^{<\kappa}$ and $Y_1\in [X]^{<\kappa}$ such that $T_0\subseteq C_1, Y_0\subseteq Y_1$ and $\langle y+K \mid y\in Y_1\rangle \cong \langle Y_1 \rangle /\langle k_j \mid j\in C_1\rangle$, and so forth.

Put $T = \bigcup _{n<\omega} T_n = \bigcup _{n<\omega} C_n$ and $Y = \bigcup _{n<\omega} Y_n$. The system $\{k_j = 0 \mid j\in T\}$ has cardinality $<\kappa$ (since $\gamma$ is uncountable) and it has no weakly nontrivial solution in $\Z$. Whence the subgroup $H = \langle y+K \mid y\in Y\rangle \cong \langle Y\rangle/\langle k_j \mid j\in T\rangle$ of $A$ belongs to $\Ker\Hom(-,\Z)$. However, this is impossible since $a_\alpha\in H$ for $\alpha<\gamma$ satisfying $|H|<\kappa_\alpha$.
\end{proof}

In the proof above, we have actually showed a little bit more. In fact, we have the following

\begin{thm} \label{t:WS_main} Let $\kappa$ be a cardinal, and assume that $\kappa$ is not at the same time singular and the least $\mathcal L_{\omega_1\omega}$-compact cardinal. The following conditions are equivalent:
\begin{enumerate}
	\item $\kappa$ is $\mathcal L_{\omega_1\omega}$-compact.
	\item Any nonempty system $S$ of homogeneous $\Z$-linear equations has a weakly nontrivial solution in $\Z$ provided that each its nonempty subsystem of cardinality $<\kappa$ has one. In other words, $\kappa\in\mathcal{WS}$.
\end{enumerate}
\end{thm}

\begin{proof} The implication `$(1)\Longrightarrow (2)$' follows immediately from `$(1)\Longrightarrow (3)$' in Proposition~\ref{p:WS_main}. The other implication then from the first part of the proof of $`\neg (1)\Longrightarrow \neg (3)$' in Proposition~\ref{p:WS_main}.

\end{proof}

\section{Concluding remarks} \label{s:conrem}

The problem of existence of $\kappa$-free groups with trivial dual turns out to be rather delicate. Under the assumption V = L (even a much weaker one), there are $\kappa$-free groups with trivial dual for any uncountable cardinal $\kappa$. Moreover, if $\kappa$ is regular and not weakly compact, then the groups can be constructed of cardinality $\kappa$, see~\cite{DG}. If $\kappa$ is singular or weakly compact, then $\kappa$-free implies $\kappa^+$-free. For more information on the topic, we refer to \cite[Chapter VII]{EM}. Anyway, we have $\mathcal{S} = \mathcal{WS} = \varnothing$ under V~=~L by Corollary~\ref{c:S_neg}.

In \cite{GS}, G\"obel and Shelah show in ZFC that $\aleph_n$-free groups with cardinality $\beth_n$ and trivial dual exist for all $0<n<\omega$. This is further generalized in \cite{Sh}\footnote{Very heavy in content. Unpublished outside arXiv.org so far.}, where Shelah proves in ZFC the existence of $\kappa$-free groups with trivial dual for any uncountable $\kappa<\aleph_{\omega_1\cdot\omega}$. On the other hand, he also shows (modulo the existence of a~supercompact cardinal) that it is relatively consistent with ZFC that there is no $\aleph_{\omega_1\cdot\omega}$-free group with trivial dual.

By Corollary~\ref{c:S_neg}, we thus know in ZFC that $\kappa\not\in\mathcal S$ for $\kappa<\aleph_{\omega_1\cdot\omega}$. However, we do not know what happens for larger cardinals $\kappa$ since the existence of a $\kappa$-free group with trivial dual is just a sufficient condition for $\kappa\not\in\mathcal S$. We have only the upper bound given by Corollary~\ref{c:S_pos}. It might still be possible that $\mathcal S = \mathcal {WS}$ where for the latter class, we have a decent description in Theorem~\ref{t:WS_main}. As shown in \cite{BM}, relative to the existence of a supercompact cardinal, there are models of ZFC where the smallest $\mathcal L_{\omega_1\omega}$-compact cardinal $\kappa$ is singular. In this only case, we cannot resolve the question whether $\kappa\in\mathcal{WS}$ although we conjecture that this is not the case, which would readily imply that at least $\mathcal{WS}\subseteq \mathcal{S}$ always holds.

A possible direction for further research is to investigate further what more can be proved in ZFC about the class $\mathcal S$.

\bigskip

\emph{Acknowledgement}: I would like to thank Petr Glivick\'y for a fruitful discussion about compactness problems in set theory.

%
%
%


\end{document}